\documentclass{amsart}
\usepackage{fontenc}
\usepackage{url}
\usepackage{hyperref}
\newtheorem{theorem}{Theorem}
\theoremstyle{definition}
\newtheorem{definition}{Definition}
\title{Quasiperfect graph}
\author{Veronica Phan*}
\thanks{*Ho Chi Minh City; email: \url{kyubivulpes@gmail.com}}
\begin{document}
\maketitle
\begin{abstract}
A perfect graph is a graph which every induced subgraph has clique number equal to chromatic number. In this paper, I will introduce a new family of graphs, the quasiperfect graphs which generalizes the perfect graphs.
\end{abstract}
\section{introduction}
All graphs in this paper are finite graphs. Let a graph $G=(V;E)$. $\omega(G),\alpha(G),\chi(G)$ be the clique number, independent number, chromatic number of $G$, respectively, $\overline{G}$ be the complement of $G$, $G[S]$ be the subgraph induced in $G$ by $S$, $K_n$ be the complete graph of order $n$, $K_0$ be the trivial graph.

A perfect graph is a graph which every induced subgraphs has clique number equal to chromatic number \cite{perfect1,perfect2}. They are important objects in graph theory. In this paper, I will introduce a new family of graphs, the quasiperfect graphs which generalizes the perfect graph and show the quasiperfect graphs also share some properties with the perfect graphs. 
\begin{definition}
\label{qp}
The definition of quasiperfect graph is inductive. Let $K_0$ be quasiperfect. Let a graph $G=(V,E)$, assume we define a graph is quasiperfect or not for all graph have less vertices than $G$. $G$ is quasiperfect if and only if:

$-$ There exist an independent set $PI$ such that it intersects all maximum cliques, each vertex of $PI$ contained in a maximum clique and $G[V-PI]$ is quasiperfect.

$-$ There exist a clique $PK$ such that it intersects all maximum independent sets, each vertex of $PK$ contained in a maximum independent set and $G[V-PK]$ is quasiperfect..
\end{definition}
We call the set $PI,PK$ as above prime independent set and prime clique respectively.
\section{The properties of quasiperfect graphs}
We want the quasiperfect graphs have the most basic property: clique number equals chromatic number.
\begin{theorem}
\label{basic}
A quasiperfect graph $G=(V,E)$ has clique number equals chromatic number.
\end{theorem}
\begin{proof}
We'll prove by induction.

It's trivial for $G=K_0$. Assume the statement is true for all induced quasiperfect subgraph of $G$. We need to prove $G$ has clique number equals chromatic number.

Let $PI$ be a prime independent set of $G$. Then we have $G[V-PI]$ is quasiperfect, so by induction hypothesis, $\omega(G[V-I])=\chi(G[V-PI])$. We have $PI$ intersects all maximum cliques, and for each cliques, there as most $1$ vertex in $PI$ because $PI$ is independent, so $\omega(G)=\omega(G[V-PI])+1$. We have a $\omega(G)$-coloring of $G$ by taking a $\omega(G[V-PI])$-coloring of $G[V-PI]$, and color all vertices in $PI$ with a new color. So $G$ clique number equals chromatic number.
\end{proof}
This theorem is the analog of \emph{Weakly perfect graph theorem}:
\begin{theorem}
\label{qpgt}
The complement of a quasiperfect graph $G=(V,E)$ is also quasiperfect.
\end{theorem}
\begin{proof}
We'll prove by induction.

It's trivial for $G=K_0$. Assume the statement is true for all induced quasiperfect subgraph of $G$. We need to prove $\overline{G}$ is quasiperfect.

The complement of a graph turn independent set into clique and vice versa. So we easily see that prime independent set of $G$ is prime clique of $\overline{G}$ and vice versa. So by Definition \ref{qp}, you need to prove that $\overline{G[V-PI]},\overline{G[V-PC]}$ is quasiperfect, which is true by induction hypothesis and Definition \ref{qp}.
\end{proof}
So we see that the quasiperfect graph have the clique number equal the chromatic number, the complement of a quasiperfect graph is also quasiperfect graph. They are also basic properties of perfect graphs as we want.
\section{Some families of graphs that are quasiperfect}
We will show the quasiperfect graphs is the non-trivial generation of the perfect graphs
\subsection{The perfect graphs}
We also hope that the perfect graphs is quasiperfect. Because all induced subgraphs of the perfect graph is also perfect, we just need to proof there exists a prime clique and prime independent set of perfect graph and the result follow by induction. By taking complement, we just need to find a prime clique. We inspired by the proof of \emph{Weakly perfect graph theorem} by Lovász \cite{wpgt1,wpgt2}.

Let $G$ be a perfect graph. For each vertex $v$ of $G$, we replace it by a clique with $t_v$ vertices to create a new graph $G'$, $t_v$ is the number of maximum independent set contains $v$, if $t_v=0$, we just delete $v$. Because all induced subgraphs of a perfect graph is also perfect and replace a vertex by a clique create a new perfect graph so $G'$ is perfect, so it have clique number equal chromatic number. By the construction of $G'$, we can correspond each vertex of $G'$ to a maximum independent set of $G$, so we have a $I$-coloring of $G'$ with $I$ is the number of maximum independent set of $G$. We see that the set of all vertices colored by the same color is a maximum independent set of $G'$, so this is an optimal coloring, so $\omega(G)=\chi(G)=I$. We take a maximum clique $K'$ of size $I$ in $G'$, then $K'$ intersects all maximum independent set of $G'$. Let $PK$ be the set of all vertices $v$ of $G$ such that there is a vertex $v'\in K'$ created by $v$, then we easily have $PK$ is a prime clique of $G$.

So the quasiperfect graphs is a generation of the perfect graphs. We can use ideas from perfect graph to work with quasiperfect graph.
\subsection{Graphs created from odd cycles}
Now we create an imperfect graph which is quasiperfect. The easiest way is creating from an odd cycles $C_n$ with $n\geq 5$. The chromatic number of $C_n$ is $3$, so we need to add cliques of size $3$. Let $v_1=v_{n+1},v_2,...,v_n$ be vertices of $C_n$, with $v_i,v_{i+1}$ are joined. We add new vertices $w_{k_1},w_{k_2},...,w_{k_o},1\leq k_p\leq n$ and join $w_{k_p}$ with $v_{k_p},v_{k_p+1}$ to create a new graph $G=(V,E)$.

We can choose prime clique $PK=\{w_{k_1},v_{k_1},v_{k_1+1}\}$ then it's easy to see that $G[V-PK]$ is a block graph (a graph which every biconnected component is a clique), so is perfect and also quasiperfect. 

Now we choose prime independent set $PI$. There are two cases:

$-$ If $o=n=2m+1$, we choose $PI=\{w_n,v_2,v_4,...,v_{2m}\}$. It's easy to see that $PI$ is a prime independent set, and $G[V-PI]$ is a forest, so is perfect and also quasiperfect.

$-$ If $o<n$, we taking the set $P=\{v_{k_1},v_{k_2},...,v_{k_o}\}$, if there exists $l$ such that $v_{l-1},v_l,v_{l+1}\in P$ and $v_l,v_{l+1}$ are joined but $v_l,v_{l-1}$ are not, delete $v_l$ from the set, in the end we have a prime independent set $PI$ (because $o<n$). It's easy to see that $G[V-PI]$ is a forest, so is perfect and also quasiperfect. 

So $G$ is quasiperfect and the quasiperfect graphs is a non-trivial generation of the perfect graphs.
\section{Further remarks}
The trick of replacing a vertex by a clique doesn't work for quasiperfect case. For example, we add a vertex to $C_5$ and join it to 2 vertices that are joined, then replace all vertices of $C_5$ by large enough cliques of the same size, it is easy to see that the result graph doesn't have clique number equal chromatic number. And we can proof that the trick only work for perfect graph, so we can't hope for any better generation.

It's a natural question to ask which induced subgraphs of a quasiperfect graph is quasiperfect? In the coloring in Theorem \ref{basic}, we see that the prime independent set is a set of all vertices have the same color, so we can hope that, give a minimum coloring and a color $c$, remove all the vertices have color $C$ from the quasiperfect creats a new quasiperfect graph.

There are some families of graphs base on algebraic objects have clique number equal chromatic number such as the enhanced power graphs of groups, the annihilating-ideal graphs of commutative rings \cite{epg1,epg2,aig}. We hope that those families of graphs are also quaiperfect.

I've showed above that there are quasiperfect graphs that have odd cycle $C_n$ with $n\geq 5$ as their induced subgraph. It's a natural question to ask for arbitrary graph $G=(V,E)$, is there a quasiperfect graphs $G'=(V',E')$ that has $G$ as its induced graph, and for fixed value of $|V|$ how smallest $|V'|$ can be?

It's likely that the properties "quasiperfect" is completely global, so by research on the quasiperfect graphs (or even classification), we can learn more about global property of the perfect graphs and other families of graphs that are quasiperfect.
\section{acknowledgement}
I would like to thank Professor Peter J. Cameron for endorsing me to submit this paper.

\end{document}